\documentclass{amsart}
\usepackage{amsmath}
\usepackage{amssymb}
\usepackage{amscd}
\usepackage{latexsym}
\usepackage{graphicx} 
\usepackage{subfigure}

\theoremstyle{plain}
\newtheorem{thm}{Theorem}
\newtheorem{lem}{Lemma}
\newtheorem{cor}{Corollary}
\newtheorem{prop}{Proposition}

\theoremstyle{definition}

\def \CPb {\overline{\mathbf{CP}}{}^{2}}
\def \CP {\mathbf{CP}{}^{2}} 
\def \CPC {\mathbf{CP}{}^{2}\# \,3\,\overline{\mathbf{CP}}{}^{2}}

\def \R {\mathbf{R}}
\def \Z {\mathbf{Z}}
\def \Sig{\Sigma}

\def \vt {\vartheta}
\def \vp {\varphi}

\def \a {\alpha}
\def \b {\beta}

\def\L{\Lambda}

\def \o {\omega}

\def \bd {\partial}
\def \x {\times}
\def \- {\smallsetminus}
\def \C {\subset}

\def \sign{{\text{sign}}}

\def \DD {\Delta}

\def \wB {\widetilde{B}}

\def\spinc{spin$^{\text{c}}$}

\def \bbF {\mathbf{F}_1}
\def \TO {T_0\x T_0}

\begin{document}

\baselineskip.5cm
\title [Surgery on Nullhomologous Tori] {Surgery on Nullhomologous Tori}
\author[Ronald Fintushel]{Ronald Fintushel}
\address{Department of Mathematics, Michigan State University \newline
\hspace*{.375in}East Lansing, Michigan 48824}
\email{\rm{ronfint@math.msu.edu}}
\thanks{R.F. was partially supported by NSF Grant DMS-1006322.}
\author[Ronald J. Stern]{Ronald J. Stern}
\address{Department of Mathematics, University of California \newline
\hspace*{.375in}Irvine,  California 92697}
\email{\rm{rstern@uci.edu}}

\begin{abstract} By studying the example of smooth structures on $\CPC$ we illustrate how surgery on a single embedded nullhomologous torus can be utilized to change the symplectic structure, the Seiberg-Witten invariant, and hence the smooth structure on a $4$-manifold.
\end{abstract}
\maketitle

\section{Introduction}\label{Intro}

Around 2007 a major breakthrough occurred in smooth $4$-manifold theory when several authors (\cite{AP} and \cite{BK}) produced what was then the smallest known simply connected $4$-manifold, $\CPC$, with more than one smooth structure. The idea of reverse engineering, as developed by the authors of this paper and presented in \cite{RevEng}, is to find a symplectic model which should have the same Euler characteristic and signature as the manifold to be constructed, and should contain enough disjoint Lagrangian tori so that surgeries on them kill $\pi_1$. In particular, for $\CPC$, we showed in 2006 that an appropriate model is $Sym^2(\Sig_3)$, the symplectic square of a Riemann surface of genus $3$, but we never thought that the result of our surgeries could ever be simply connected. However, it turns out that the manifolds produced in \cite{AP} and \cite{BK} can equivalently be obtained from reverse engineering starting with $Sym^2(\Sig_3)$, and, from our vantage point, proving simple connectivity was a major contribution.

Pushing this further, as pointed out in \cite{RevEng} and \cite{Pin}, this procedure can be reinterpreted to show that surgery on nullhomologous tori is the underlying operation that is used to change the smooth structure of a $4$-manifold.  In this paper we hope to motivate and enhance this idea by proving the following theorem.

\begin{thm}\label{one} There is a nullhomologous torus $T$ embedded in $\CPC$ such that surgeries on $T$ result in an infinite family of smooth $4$-manifolds which are all homeomorphic to $\CPC$ but which are pairwise nondiffeomorphic.
\end{thm}

\noindent This theorem is true more generally for $\mathbf{CP}{}^{2}\# \,k\,\overline{\mathbf{CP}}{}^{2}$, $2\le k\le 7$, following the ideas below and the constructions of \cite{Pin}.

Given a general smooth $4$-manifold, it is not known whether it is possible to find such tori, and even if it is, they are difficult to find. We give a hint at their detection by showing below that in the case of symplectic manifolds with $b^+=1$ these nullhomologous tori are constrained by their relationship with the canonical class.

\section{Nullhomologous tori}

An effective procedure for altering the smooth structure of a given 4-manifold X relies on the ability to find a suitably embedded torus $T$ with trivial normal bundle that represents a nontrivial homology class  and to perform surgeries on this torus. If the torus $T$ is the fiber of an elliptic fibration, this surgery is also known as `log transform'.  If both $X$ and $X\- N_T$ are simply connected, and if the Seiberg-Witten invariant of $X$ is nontrivial, then another effective surgery on $T$, the knot surgery introduced in \cite{KL4M}, will also change the smooth structure of $X$. In each of these cases, the new smooth manifold also contains an essential torus of self-intersection $0$. However, not all $4$-manifolds contain such suitably embedded tori. In particular, the following lemma shows that in the situation which is of particular interest to us, we can find no tori with trivial normal bundle representing a nontrivial homology class. 

\begin{lem}\label{null} Let $X$ be a smooth simply connected $4$-manifold with $b^+=1$, $b^-\le8$ and with a nontrivial Seiberg-Witten invariant. Then there are no homologically essential embedded tori of square $0$ in $X$.
\end{lem}

\begin{proof} Let $k\in H_2(X;\Z)$ be a basic class, i.e $k$ is the Poincar\'e dual of $c_1(s)$ where $s$ is a \spinc-structure for which the Seiberg-Witten invariant is nonzero. 
Let $\{h,e_1,\dots,e_b\}$ be an orthogonal basis for $H_2(X;\Z)$ where $h^2=1$ and $e_i^2=-1$. Write 
$T= ah-\sum b_ie_i$ and $k=\a h-\sum \b_ie_i$. Then $T^2=0$, and since the dimension of the Seiberg-Witten moduli space corresponding to $k$ is nonnegative, we also have  $k^2\ge(3\,\sign+2\,e)(X)$, and this is $>0$ because of our hypothesis on $X$. Hence
\[ a^2=\sum b_i^2, \ \ \  \a^2>\sum \b_i^2\]
It follows directly from the adjunction inequality that $k\cdot T=0$; hence $a\a=\sum b_i\b_i$.
Thus Cauchy-Schwartz implies
\[ a\a > \sqrt{\sum b_i^2}\sqrt{\sum \b_i^2} \ge |\sum b_i\b_i | = |a\a| \]
This contradiction proves the lemma.
\end{proof}

In summary, for some manifolds, in particular the manifolds of the important class above, there are no essential tori upon which surgery can be used to change the smooth structure. Thus we are led to search for interesting nullhomologous tori in these manifolds.

\section{Surgery on tori}

Suppose that $T$ is a torus of self-intersection $0$ in a $4$-manifold $X$ with tubular neighborhood $N_T$. Let $a$ and $b$ be generators of $\pi_1(T^2)$ and  let $S^1_{a}$ and $S^1_{b}$ be loops in $T^3=\bd N_T$ homologous in $N_T$ to $a$ and $b$ respectively, and let $\mu_T$ denote a meridional circle to $T$ lying on $T^3$. By {\it{$p/q$-surgery}} on $T$ with respect to $b$ we mean 
\begin{equation} \label{surgery}
\begin{split} X_{T,b}&(p/q) =(X\- N_T) \cup_{\vp} (S^1\x S^1 \x D^2),\\
& \vp: S^1 \x S^1 \x \bd D^2 \to \bd(X\- N_T) \end{split}
\end{equation}
where the gluing map 
satisfies $\vp_*([\bd D^2]) = q [S^1_{b}] + p [\mu_T]$ in $H_1(\bd(X\- N_T);\Z )$.  When $T$ and $b$ are understood, we simply write $X_{p/q}$.

Note that we have framed $N_{T}$ using $S^1_{a}$ and $S^1_{b}$ so that pushoffs of $a$ and $b$ in this framing are $S^1_{a}$ and $S^{1}_{b}$. When the curve $S^{1}_{b}$ is nullhomologous in $X\- N_T$, then $H_1(X_{T,b}(1/q);\Z )= H_1(X;\Z )$. As long as $\mu_T$ generates an infinite cyclic summand of $H_1(X\- N_T;\Z )$, the framing  $S^1_b$ is the unique ``nullhomologous framing'' of $b$.
If in addition, $T$ itself is nullhomologous,  then $H_1(X_{T,\b}(p/q);\Z )= H_1(X;\Z ) \oplus \Z / p\Z$.

Let $T'$ be a homologically primitive torus in a $4$-manifold $X'$ with tubular neighborhood $N_{T'}$, and suppose that $b'$ is a nonseparating loop on $T'$. Because $T'$ is primitive, there is an oriented surface which intersects $T'$ once. Thus the meridian $\mu_{T'}$ bounds a surface in $X'\- N_{T'}$. Fix a framing of $T'$ and suppose that the pushoff $S^1_{b'}$ represents a nontrivial element of $H_1(X'\- N_{T'};\R)$.

Now perform $+1$-surgery on $T'$ with respect to $b'$ and the framing that was chosen above. (The situation is the same for $-1$ surgery on $T'$.) The manifold $X$ which results from this surgery contains a torus $T$ of self-intersection $0$, the core torus 
$S^1\x S^1 \x \{0\}$ in the formula above. The meridional loop $\mu_T$ to this torus is homologous to $S^1_{b'}+\mu_{T'}$, which is in turn homologous to $S^1_{b'}$, since $\mu_{T'}$ is homologically trivial in $X'\- N_{T'}=X\- N_T$. Thus $\mu_T$ represents a nonzero class in $H_1(X\- N_T;\R)$; so $T$ must be nullhomologous. The gluing \eqref{surgery} takes  $\mu_{T'}$ to a loop $S^1_b$ on $\bd N_T$ which is homotopic in $N_T$ to a loop $b$ on $T$. We see that the pushoff $S^1_b$ is nullhomologous in $X\- N_T$.
Conversely if we perform $0$-surgery on $T$ with respect to $b$ and the framing given by $S^1_b=\mu_{T'}$, the resultant manifold is $X'$. The situation is summarized in diagram \eqref{diagram}. 

\begin{equation}\label{diagram}
\begin{array}{ccc}
X'\hspace{.75in}&\xrightarrow{\text{{$+1$\ surgery}}}\hspace{.25in} & X\\
&\hspace*{-.25in}\xleftarrow{\text{{$0$\ surgery}}}&\\ 
{\hspace*{-.25in}\text{$T'$ primitive}}   && {\text{$T$ nullhomologous}} \\
{\text{$b'$ essential in complement}}&&{\text{$b$ nullhomologous in complement}}\\
& b^+(X')= b^+(X) +1&\\
\end{array}
\end{equation}

There is a useful formula of Morgan, Mrowka, and Szabo which determines Seiberg-Witten invariants in this situation. Suppose we are given a manifold $X$ upon which we would like to do $1/n$ surgery with respect to a torus $T$ and a loop $b$ as above. The Morgan-Mrowka-Szabo formula determines the Seiberg-Witten invariants of the surgered manifold $X_{1/n}$ in terms of those of $X$ itself and of $0$-surgery $X_0$ as above. We write $T_0$ for the core torus of the surgery in $X_0$. (This is the torus denoted $T'$ in the paragraph above.) Similarly $T_{1/n}$ is the core torus in $X_{1/n}$.

Let $k_0\in H_2(X_0;\Z)$ be a basic class. The adjunction inequality implies that $k_0$ is orthogonal to $T_0$. Thus, there are (unique, because $T$ (resp. $T_{1/n}$) are nullhomologous) corresponding homology classes $k_{1/n}$ and $k$ in $H_2(X_{1/n};\Z)$ and $H_2(X;\Z)$, respectively, where $k$
 agrees with the restriction of $k_0$ in $H_2(X\- N_T,\bd;\Z)$ in the diagram:
\[ \begin{array}{ccc}
H_2(X;\Z) &\longrightarrow & H_2(X, N_T;\Z)\\
&&\Big\downarrow \cong\\
&&H_2(X\- N_T,\bd;\Z)\\
&&\Big\uparrow \cong\\
H_2(X_0;\Z)&\longrightarrow & H_2(X_0,N_T;\Z)
\end{array}\]
and similarly for $k_{1/n}$.

It follows from \cite{MMS} that 
\begin{equation}\label{MMS}
SW_{X_{1/n}}(k_{1/n}) = SW_X(k) + n\sum_i SW_{X_0}(k_0+i[T_0])
\end{equation}
and that these comprise all the basic classes of $X_{1/n}$. 

If there is a torus $T_d$ of self-intersection $0$ in $X_0$ such that $T_d\cdot T_0=1$, then adjunction inequality arguments show that the sum on the right of the above formula has at most one nonzero term. We will often be in the situation where each of the manifolds $X$, $X_{1/n}$, and $X_0$ has a unique basic class up to sign. In either of these cases we can unambiguously write $SW_{X_{1/n}}=SW_X+n\,SW_{X_0}$. We see that if $SW_{X_0}\ne0$, we obtain an infinite family of distinct smooth manifolds (determined by their Seiberg-Witten invariants) all with the same homology groups as $X$.

\section{Reverse engineering}\label{RE}

We begin by discussing how one ensures that the term $SW_{X_0}$ of \eqref{MMS} is nonzero.
If $X$ is a symplectic manifold and $T$ is  any Lagrangian torus, then there is a canonical framing, called the Lagrangian framing, of $N_{T}$. This framing is uniquely determined by the property that pushoffs of $T$ in this framing remain Lagrangian. If one performs $1/n$ surgeries with respect to the pushoff in this framing of any curve on $T$, then the result is also a symplectic manifold. We refer the reader to  \cite{ADK} for a full discussion of this phenomena, which is referred to as {\it{Luttinger surgery}}. However, one must be careful that it is often the case that  the pushoff of a curve using the Lagrangian framing may not be nullhomologous, so that a $1/n$ (Luttinger) surgery may in fact may change $H_{1}$. This will be the case when performing reverse engineering. 

It is important to note that surgery on tori changes neither the signature, $\sign (X)$ nor the Euler characteristic, $e(X)$, of a $4$-manifold $X$. Also, as mentioned above, if $H_1(X;\Z)=0$ and $T$ is a nullhomologous torus with a nonseparating loop $b$ which is nullhomologous in 
$H_1(X\- N_T;\Z)\cong \Z$, then a $1/n$ surgery on $T$ with respect to the nullhomologous framing of $b$ leaves $H_1(X;\Z)$ unchanged.
Furthermore, if $X$ is simply connected, $SW_{X_0}\ne0$, and if it can also be seen that the manifolds $X_{1/n}$ are simply connected, then we obtain an infinite family of distinct manifolds all homeomorphic to $X$. This is the key idea of ``reverse engineering". The method is as follows:

Suppose that we wish to construct manifolds which are homeomorphic to a given simply connected $4$-manifold $R$ with $b^+_R\ge 1$ but which are pairwise nondiffeomorphic. Here is a recipe: First seek out a symplectic manifold $M$ which 
\begin{enumerate}
\item[1.] has the same Euler characteristic and signature as $R$. 
\item[2.] has $n=b_1(M)$ pairwise disjoint Lagrangian tori, $T_i$, each containing a nonseparating loop $b_i$, such that the collection of $b_i$ represents a basis for $H_1(M;\R)$.
\end{enumerate}
Now perform $\pm 1$ Luttinger surgeries consecutively on the Lagrangian tori $T_i$, obtaining symplectic manifolds $M_i$ with  $b_1(M_i)=b_1(M)-i$, $b^+(M_i)=b^+(M)-i$, and in particular, $b_1(M_n)=0=b_1(R)$ and $b^+(M_n)=b^+(R)$. 

Denote the symplectic manifold $M_n$ by $X$. Referring to diagram \eqref{diagram} we see that there is a distinguished nullhomologous torus $T$ in $X$ containing a loop $b$ which has a pushoff which is nullhomologous in the complement of $T$. Performing $0$-surgery with respect to this framing yields $M_{n-1}$ as in \eqref{diagram}. Because $M_{n-1}$ is symplectic and $b^+(M_{n-1})=b^+(R)+1\ge2$, it follows that $SW_{X_0}=SW_ {M_{n-1}}\ne0$.

The Seiberg-Witten invariant of $X_{1/n}$, the result of $1/n$ surgery on $T$ with respect to the nullhomologus framing of $b$, is calculated via \eqref{MMS}. For example, if we are in one of the situations where the sum on the right of  \eqref{MMS} reduces to a single term, then it is clear that the manifolds $X_{1/n}$ are pairwise nondiffeomorphic. In any case, infinitely many of the $X_{1/n}$ are pairwise nondiffeomorphic (see \cite{RevEng}). If, furthermore, each of the manifolds $X_{1/n}$ is simply connected, then they are all homeomorphic to $R$. However, it is often difficult to verify that any of the manifolds $X_{1/n}$ are indeed simply connected.

As we have already mentioned, this process is called ``reverse engineering". The manifold $M$ is called a ``model manifold" for $R$. 

\section{$\CPC$}

The reverse engineering process follows easily when $R=\CPC$. As we showed in \cite{RevEng}, $M=Sym^2(\Sig_3)$ may be taken as a model manifold in that case. Then $\pi_1(M)=\Z^6$ and one can find six Lagrangian tori upon which Luttinger surgeries kill $\pi_1$. We are then left with a symplectic manifold $X$ which is homeomorphic to $\CPC$ and a nullhomologous torus $T$ in $X$ containing a nullhomologous loop $b$ so that $1/n$ surgeries with respect to the nullhomologous framing of $b$ give an infinite family of homeomorphic but pairwise nondiffeomorphic  $4$-manifolds. Of these manifolds $X_{1/n}$, only $X$ itself is symplectic.

In fact, $X$ is not diffeomorphic to $\CPC$. That one is creating an exotic smooth structure follows almost directly from \cite{ADK}. Since $Sym^2(\Sig_3)$ is a complex surface of general type, its canonical class pairs positively with its Kahler form. Surgering Lagrangian tori does not change this phenomenon  (see \cite[p.2112]{RevEng}); so we obtain a symplectic $4$-manifold $(X,\o_X)$ which is homeomorphic to $\CPC$ but for which $K_X\cdot\o_X>0$. According to \cite{LL}, however, each symplectic structure on $\CPC$ satisfies $K\cdot\o<0$; so $X$ cannot be diffeomorphic to $\CPC$. In fact, in \cite{RevEng} it is shown by using adjunction inequality arguments that the Seiberg-Witten invariant of $X$ is $t-t^{-1}$ where $t\in \Z H_2(X;\Z)$ corresponds to the canonical class, whereas the Seiberg-Witten invariant of $\CPC$ vanishes since $\CPC$ admits a positive scalar curvature metric. 

Our goal now is to enhance ths procedure to show that there is a single nullhomologous torus $T$ in $\CPC$ upon which the sequence of  $1/n$ surgeries gives our infinite family of manifolds. To do this we need to answer two questions:
\begin{enumerate}
\item[1.] How does one go about finding nullhomologous tori in a manifold?\\
\item[2.] What criteria do they need to satisfy in order to be useful for changing the smooth structure? 
\end{enumerate}
We are not able to answer these questions generally, but we can give satisfying answers in the case of $\CPC$ and more generally for $\mathbf{CP}{}^{2}\# \,k\,\overline{\mathbf{CP}}{}^{2}$, $2\le k\le 7$.

To address the first of the questions, suppose that we have a smooth $4$-manifold $R$ which contains an embedded smooth torus $T$ of self-intersection $0$. Choose local coordinates in which a tubular neighborhood $T\x D^2$ of $T$ is $S^1\x (S^1\x D^2)$. The {\it Bing double} $B_T$ of $T$ consists of the pair of tori $S^1\x$ (Bing double of the core circle $S^1\x \{0\}$). The solid torus $S^1\x D^2$ is shown in Figure~\ref{Doubles}(a). 
\begin{figure}[ht]
\begin{center}
     \subfigure[Bing double]{\includegraphics[scale=.7]{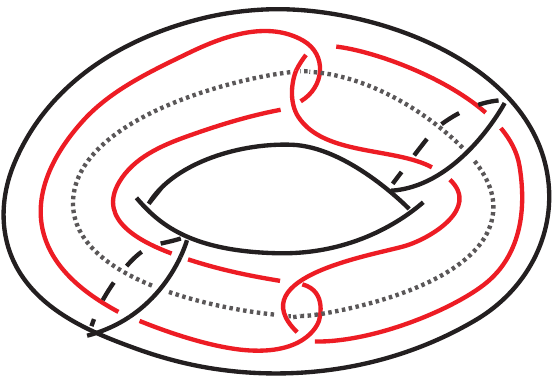}}
     \hspace{.2in}
    \subfigure[Whitehead double]{\includegraphics[scale=.7]{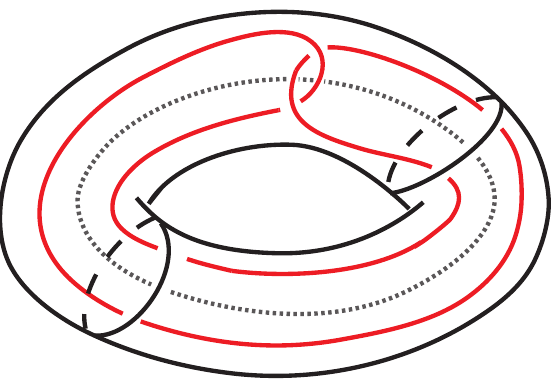}}
      \end{center}
\vspace*{-.1in}     \caption{}
  \label{Doubles}
\end{figure}
This description (including the splitting of $T^2$ into the product $S^1\x S^1$ and a fixed framing, {\it i.e.} a fixed trivialization of the normal bundle of $T$) determines this pair of tori up to isotopy. The component tori in $B_T$ are nullhomologous in $T\x D^2$. If one does $\pm 1$ surgery to one of the two tori of $B_T$, it turns the other into the Whitehead double of $T$ in $T\x D^2$. Often, surgery on the other torus will then change the smooth structure of the ambient manifold. For example, if $T$ is an elliptic fiber in the rational elliptic surface $E(1)$, a $1/n$ surgery on the second torus (with respect to an appropriate framing) will give the result of knot surgery using an $n$-twist knot (\cite{KL4M})

As we have seen in Lemma \ref{null}, one can not count on finding tori such as $T$. Instead, we seek a manifold smaller than $T^2\x D^2$ which still contains the Bing tori $B_T$. If we view $B_T\C S^1\x (S^1\x D^2)$, and write the first $S^1$ as $I_1\cup I_2$ ($I_1\cap I_2\cong S^0$), and the second $S^1$ as $J_1\cup J_2$ and consider $T_0 = S^1\x S^1\- (I_2\x J_2)$ (see Figure~2), then $B_T\cap (T_0\x D^2)$ consists of a pair of punctured tori, and $\bd(B_T\cap (T_0\x D^2))$ is a link in $\bd(T_0\x D^2) \cong S^1\!\x S^2\, \# \, S^1\!\x S^2$.

\begin{figure}[ht]
\begin{center}\includegraphics[scale=.8]{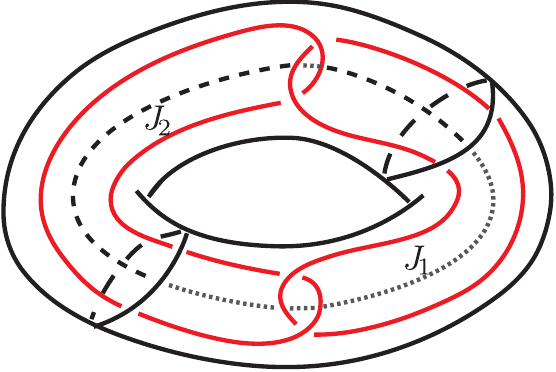}
\end{center}\vspace*{-.1in}
\caption{}
\label{J1J2}
\end{figure}

The intersection of $B_T$ with $(I_2\x J_2)\x D^2$ is a pair of disks $I_2\x$ (intersection of the Bing double link with $J_2\x D^2$). Its boundary is the double of the intersection of the Bing double link with $J_2\x D^2$; i.e. the (1-dimensional) Bing double of $\bd(I_2\x J_2\x \{0\})$. In $\bd(T_0\x D^2) \cong S^1\!\x S^2\, \# \, S^1\!\x S^2$ this boundary is shown in Figure~\ref{bdry}(a) or equivalently, Figure~\ref{bdry}(b).

\begin{figure}[ht]
\begin{center}
     \subfigure[]{\includegraphics[scale=.3]{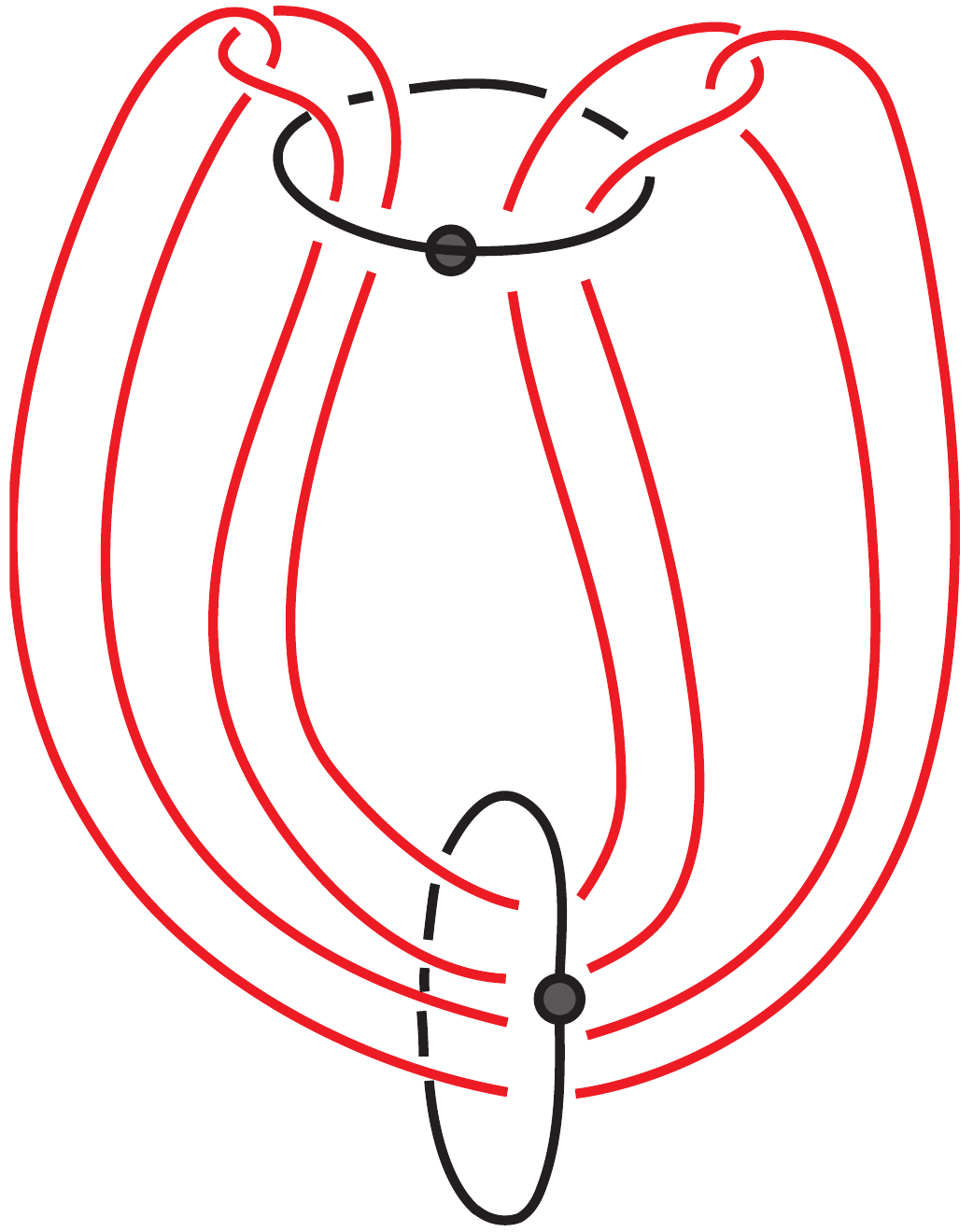}}
     \hspace{.5in}
    \subfigure[]{\includegraphics[scale=.3]{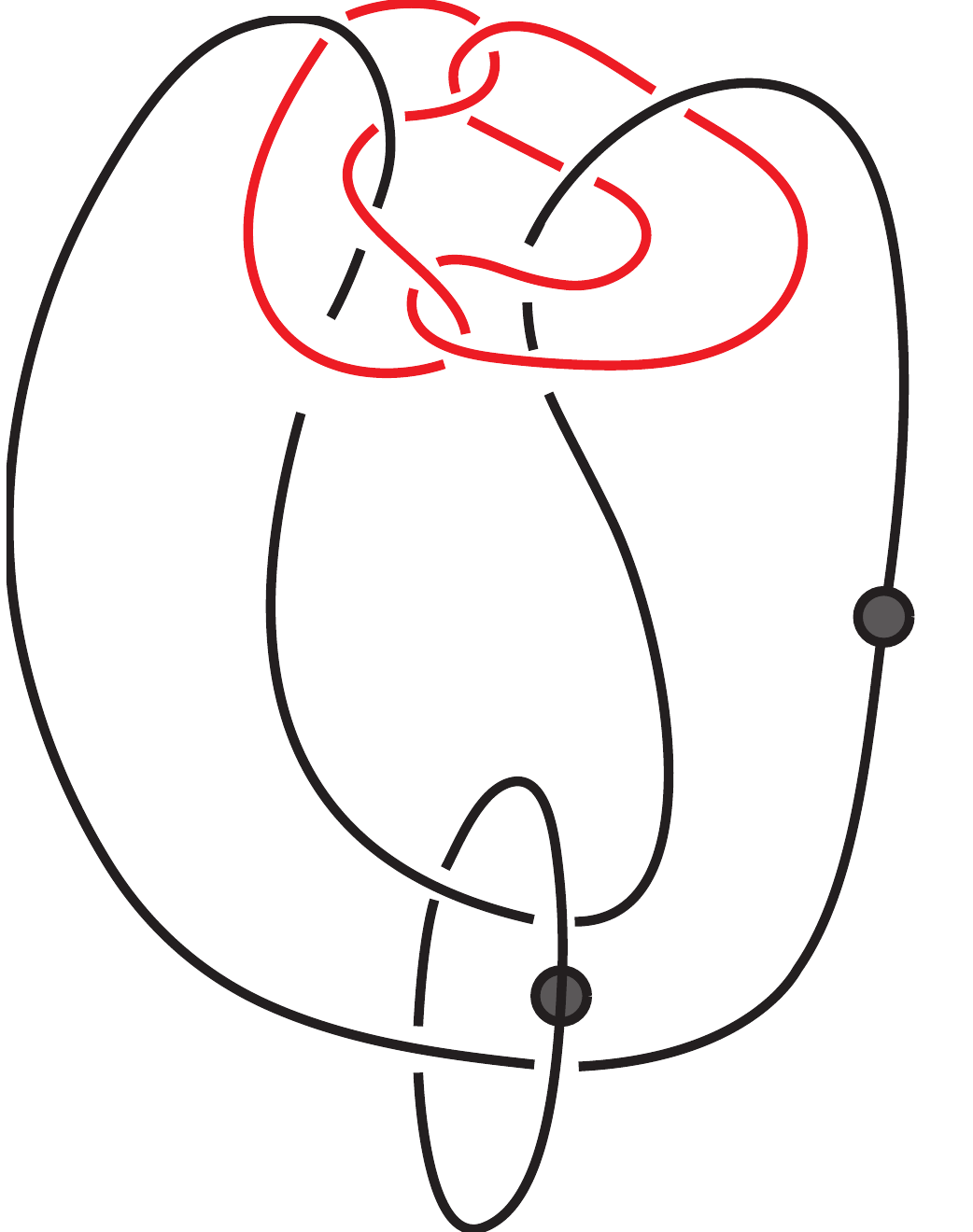}}
      \end{center}
\vspace*{-.2in}     \caption{}
  \label{bdry}
\end{figure}

Performing $0$-framed surgeries on these boundary circles (with respect to the framing in Figure~3), we obtain a manifold, shown in Figure~\ref{A}(a) which contains a pair of self-intersection $0$ tori. We call this manifold $A$. It is given equivalently by Figure~\ref{A}(b). Figure~\ref{A}(b) points out that $A$ is obtained from the $4$-ball by attaching a pair of $2$-handles and then carving out a pair of $2$-handles. The Euler characteristic of $A$ is $e=1$ and its betti numbers are $b_1(A)=2=b_2(A)$.

\begin{figure}[ht]
\begin{center}
     \subfigure[]{\includegraphics[scale=.3]{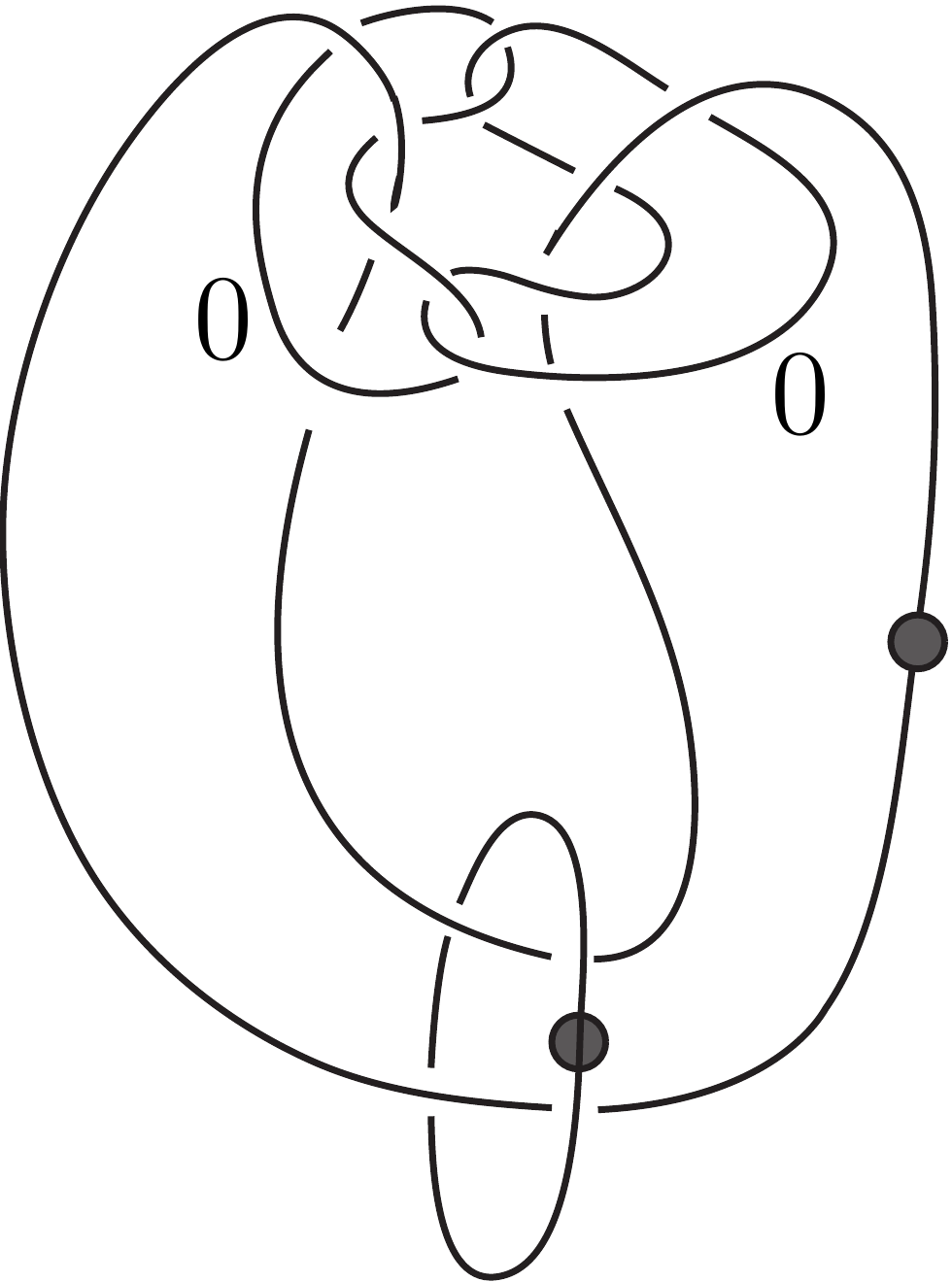}}
     \hspace{.3in}
    \subfigure[]{\includegraphics[scale=1.5]{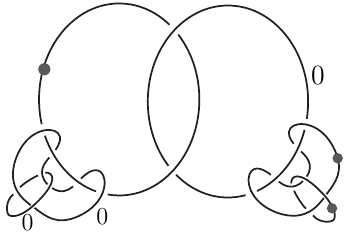}}
      \end{center}
\vspace*{-.1in}     \caption{The manifold $A$}
  \label{A}
\end{figure}

Our above discussion shows that $A$ embeds in $T^2\x D^2$ and contains the Bing double $B_T$ of the core torus. We continue to refer to the pair of tori in $A$ as $B_T$, even though, in general, they do not constitute a Bing double of some other torus.

\begin{lem}\label{SxT} The manifold $A$ embeds in $T^2\x S^2$ as the complement of a pair of transversely intersecting tori of self-intersection $0$. 
\end{lem}
\begin{proof} Write $T^2\x S^2$ as $(S^1\x (S^1\x D^2))\cup (S^1\x (S^1\x D^2))$. We know that $A$ embeds in, say, the second $T^2\x D^2$. In Figure~\ref{A}(a), if we remove the two $2$-handles, we obtain $T_0\x D^2$. The two $2$-handles are attached along the Bing double of the circle $\b=\bd T_0\x\{ 0\}$. If instead, we attach a $0$-framed $2$-handle along $\b$ we obtain $T^2\x D^2$. This implies that $(T^2\x D^2)\- A$ is the complement in the $2$-handle, $D^2\x D^2$, of the core disks of the $2$-handles attached to obtain $A$. This complement is thus the result of attaching two $1$-handles to the $4$-ball. This is precisely the boundary connected sum of two copies of $S^1\x B^3$, {\it i.e.} $T_0\x D^2$. Using the notation in Figure~\ref{J1J2} and above, $(T^2\x D^2)\- (T_0\x D^2)=I_2\x J_2\x D^2$. The complement of the two $2$-handles dug out of this is a neighborhood of $\{pt\}\x$ the shaded punctured torus in Figure~\ref{S^2xT^2}(b). 

Thus the two tori referred to in the lemma are illustrated in Figure~\ref{S^2xT^2}. One of these tori, $T$, is $S^1$ times the core circle in Figure~5(a), and the second torus, $S_T = D_T\cup T_0'$, where $D_T$ is $\{pt\}$ times the shaded meridional disk in Figure~\ref{S^2xT^2}(a), and $T_0'$ is $\{pt\}$ times the shaded punctured torus in Figure~\ref{S^2xT^2}(b). Note that $S_T$ represents the homology class of $\{ pt\}\x S^2$.

\begin{figure}[ht]
\begin{center}\includegraphics[scale=1.5]{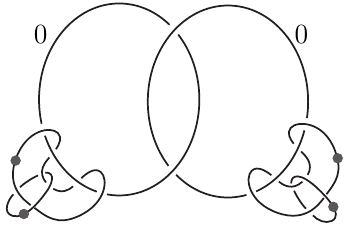}
\end{center}\vspace*{-.1in}
\caption{}
\label{N}
\end{figure}

From a Kirby calculus point of view, a depiction of a neighborhood $N$ of these two tori is shown in Figure~\ref{N}. Take its union with $A$ as seen in Figure~\ref{A}(b). The Borromean triple on the left side of Figure~\ref{A}(b) cancels with the corresponding triple in Figure~\ref{N}). We are left with the double of $T^2\x D^2$, i.e. $T^2\x S^2$. 
 \end{proof}

\begin{figure}[ht]
\begin{center}
     \subfigure[]{\includegraphics[scale=.95]{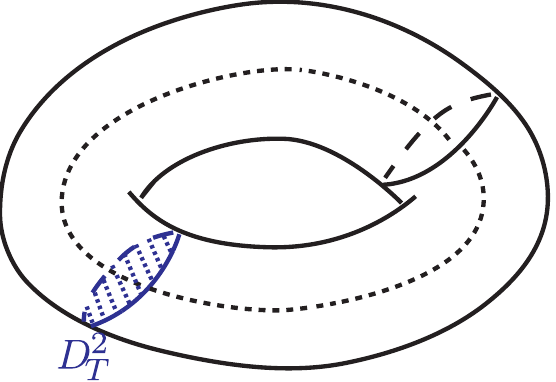}}
     \hspace{.2in}
    \subfigure[]{\includegraphics[scale=.95]{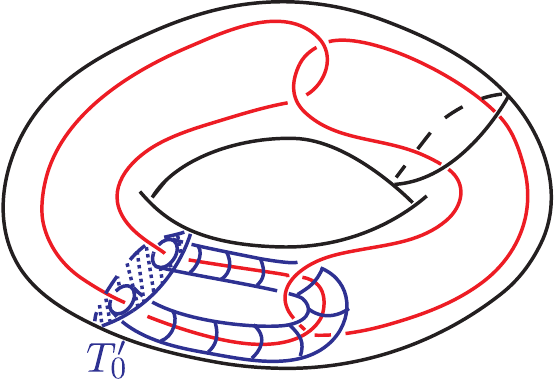}}
      \end{center}
\vspace*{-.1in}     \caption{}
  \label{S^2xT^2}
\end{figure}

\begin{cor}\label{cor1} The manifold $A$ is the complement of  the punctured torus $S_T\- D^2 =\{ pt \} \x T_0'$ in $S^1\x S^1 \x D^2$.
\qed
\end{cor}

In Figure~\ref{ALabelled} we again see $A$, and using this figure, we can describe the tori in $B_T$. It shows the Borromean $2$-handles whose cocores have boundary circles $x$ and $y$. Denote disks which comprise the cores of these handles by $D_x$ and $D_y$.
Figure~\ref{ALabelled} also shows the circles $a$ and $b$ which go over the Borromean $1$-handles of $A$. Let $b_1^*$ and $b_2^*$ be the components of the Bing double of $b$.  Let $T_{(1)}$ and $T_{(2)}$ be the component tori of $B_T$. Then we have:

\begin{lem} \label{DescribeBT} The torus $T_{(i)}$ can be decomposed in the standard way into a union of a $0$-cell, two $1$-cells, and a $2$-cell, where the union of the $0$ and $1$-cells deformation retracts to a wedge of circles $C_1\vee C_2$ where $C_1$ is isotopic in $A$ to $a$ and $C_2$ is isotopic to $b_i^*$. The $2$-cell of $T_{(i)}$ is $D_x \ (i=1)$ or $D_y\  (i=2)$. \qed
\end{lem}

\begin{figure}[ht]
\begin{center}\includegraphics[scale=1.5]{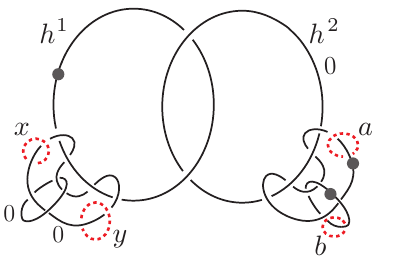}
\end{center}\vspace*{-.1in}
\caption{}
\label{ALabelled}
\end{figure}

It is also useful to interpret Corollary~\ref{cor1}  in light of Figure \ref{A}(b). In that figure, the large $2$-handle together with the Bing double $1$-handles give a handle description of $T^2\x D^2$. If we remove a $D^2$-fiber (thus obtaining $T_0\x D^2$) we add the large $1$-handle as in Figure \ref{A}(b). However, this is not what we have done --- instead we have removed a punctured torus. This means that we have removed a Bing double pair of $1$-handles which is the same as attaching the Bing double pair of $2$-handles in Figure \ref{A}(b).

Note that we see $B_T\C S^1\x$ (the solid torus in Figure~\ref{S^2xT^2}(b)) $=T^2\x D^2\C T^2\x S^2$. View $S^1\x S^2$ as $0$-framed surgery on an unknot in $S^3$. The Bing double of the core circle in Figure~\ref{S^2xT^2}(b) is the Bing double of the meridian to the $0$-framed unknot. Performing $0$-framed surgery on the two components of this Bing double gives us $\{0,0,0\}$-surgery on the Borromean rings, {\it viz.} $T^3$. Thus performing $S^1$ times these surgeries gives:

\begin{prop}\label{standard} One can perform surgery on the tori $B_T\C A\C T^2\x S^2$ to obtain the $4$-torus, $T^4$. \qed
\end{prop}

Later we will be interested in other surgeries on $B_T$. We call the surgeries of Proposition~\ref{standard} the {\it standard} surgeries on $B_T$.
Conversely, {\it standard} surgeries on the corresponding pair of tori $\wB_T\C T^4$ yields $T^2\x S^2$. Furthermore, $\wB_T$ is a pair of disjoint Lagrangian tori in $T^4$, $S^1$ times two of the generating circles of $T^3$. The pair of tori of Lemma~\ref{SxT} can also be identified in $T^4$ after the surgeries. The first torus, $S^1$ times the core circle in Figure~5(a) becomes $S^1$ times the third generating circle of $T^3$. Call this torus $T_T$. The other torus intersects $T_T$ once and is disjoint from $B_T$. We call it $T_S$. It is the dual generating torus of $T^4$.  The complement of these tori in $T^4$ is $T_0\x T_0$. We thus have:

\begin{prop}\label{T_0->A} The standard surgeries on the pair of Lagrangian tori $\wB_T$ in $\TO$ give rise to $A$, and conversely, the standard surgeries on $B_T\C A$ yield $\TO$.\qed
\end{prop}

Thus the result of the standard surgeries on $T^2\x S^2$ is to transform $A$ into the complement of a transverse pair of generating tori $T_T=T^2\x\{pt\}$ and $T_S=\{pt\}\x T^2$ in $T^4$. The reason for this notation is that $T_S$ is the torus in $T^4$ which is sent to $S_T$ in $T^2\x S^2$ after standard surgeries on $\wB_T$, and $T_T$ is the torus that is sent to $T$.

These surgeries also transform the Bing tori in $B_T$ into the Lagrangian tori $\L_1 = S^1_1\x S^1_3$ and $\L_2 = S^1_1\x S^1_4$ in $T^4=T^2\x T^2=(S^1_1\x S^1_2)\x (S^1_3\x S^1_4)$. The surgeries on $\L_1$ and $\L_2$ are not Lagrangian surgeries in the sense of \cite{ADK}, and so one does not get an induced symplectic structure. Indeed, $T_0\x T_0$ is the complement of transversely intersecting symplectic tori in $T^4$, but after surgery, in $T^2\x S^2$, the complement of $A$ is the regular neighborhood of a pair of tori, one of which is not minimal genus and so cannot be symplectically embedded.

\section{The canonical class}

We now begin  to address question 2 from the previous section. What constraints are placed on a nullhomologous torus (or collection of nullhomologous tori) if they are to be useful for changing the smooth structure? We  limit our discussion to the case of rational surfaces $R=\mathbf{CP}{}^{2}\# \,k\,\overline{\mathbf{CP}}{}^{2}$, $k\le 8$. As mentioned above, for these manifolds Li and Liu have shown that the symplectic form is unique up to deformation and diffeomorphism, and hence for any symplectic form $\o_R$, the corresponding canonical class $K_R$ must satisfy $K_R\cdot\o_R<0$ (as is true for the standard Kahler form) \cite{LL}. Thus our goal should be to find nullhomologous tori in $R$ such that surgeries on them produce a symplectic manifold $(X,\o_X)$ homeomorphic to $R$ with $K_X\cdot \o_X >0$. Then $X$ cannot be diffeomorphic to $R$. One important constraint can be derived from the genus of a surface representing $K_X$.

Since $K_R$ is a negative class (i.e. it intersects $\o_R$ negatively) $K_R$ is not represented by a symplectic surface. Rather,  it is $-K_R=3h-\sum e_i$ that is. The adjunction formula then applies to $-K_R$ and states that for $g$ the genus of $\pm K_R$,
\[ 2g-2=(-K_R)\cdot(-K_R) + (-K_R)\cdot K_R \]
so $g=1$, as we know. If we can find a symplectic structure $\o_X$ as above, then $K_X$ is positive and will be represented by a symplectic surface (cf. \cite{T} and \cite{LL}). Thus we can apply the adjunction formula to $K_X$:
\[ 2g-2=K_X\cdot K_X +K_X\cdot K_X \]
and we see that $K_X$ is represented by symplectic surface of genus 
\[ g=K_X^2 +1= c_1^2(X) +1 = 10-k\]
In particular, for $R=\CPC$, $k=3$ and $g=7$. Thus if we have a collection of nullhomologous tori such that surgery on them changes $R$ to $X$ they must serve to increase the genus of the canonical class by $6$. Notice that this is precisely the number of surgeries needed to go from the model $Sym^2(\Sig_3)$ to $X$!

\section {Pinwheels}

We next address the question of how one can go about finding nullhomologous tori in $\CPC$ which will serve to raise the genus of the canonical class from $1$ to $7$. As a warmup, referring to Figure \ref{S^2xT^2}(b), note that surgery on one of the Bing tori intersecting a  disk $\{pt\}\x \{pt\}\x D^2\C S^1\x S^1\x D^2$ will not embed after the surgery, however there is an obvious puntured torus, $T_0'$ which will. Thus, surgery will force the genus of this disk to go up by one. The technique presented in \cite{Pin} will be useful for embedding nullhomologous tori so that surgery on them accomplishes this. Rather than repeating what was already presented in that paper, we will illustrate the idea of a pinwheel structure in the case of $\CP$ and $\CPC$ and refer the reader to \cite{Pin} for specifics and generalizations.

Toward this end consider a linear $T^2$-action on $\CP$. Its orbit space is a polygon as shown in Figure \ref{CP}(a). The vertices are images of fixed points, while an edge labeled $(p,q)$ has isotropy group the circle group described in polar coordinates by $G(p,q)=\{(\vp,\vt)\mid p\vp+q\vt=0, gcd(p,q)=1\}$, and the interior points are the images of principal orbits. We have divided this figure into three quadrilaterals; each is the image of a standard $4$-ball coordinate neighborhood of one of the three fixed points of this torus action. In Figure \ref{CP}(a) we have also noted the self-intersection numbers (all $+1$) of the $2$-spheres described by the edges.

\begin{figure}[ht]
\begin{center}
    \subfigure[$\CP$]{\includegraphics[scale=.9]{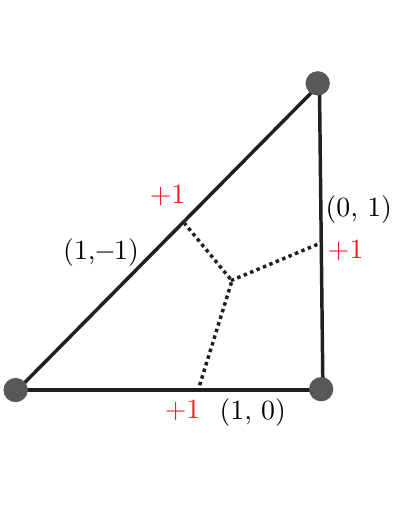}} \hspace{.25in}
    \subfigure[$\CP$ with pinwheel structure]{\includegraphics[scale=.9]{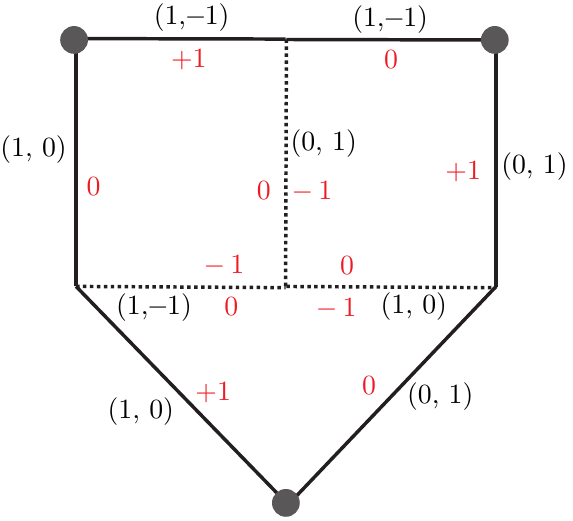}}
      \end{center}
\vspace*{-.1in}     \caption{}
  \label{CP}
\end{figure}

Let $\bbF$ be the rational ruled surface with a negative section $S_-$ of square $-1$. It is useful to view each of these $4$-balls as the complement in $\bbF$ of $S_-$ and a fiber $F$.
Figure \ref{CP}(b) indicates how these complements are glued together. (This example was originally described in a slightly different fashion in \cite{S}.) The extra labelings in this figure give the isotropy types and self-intersection numbers for the spheres $S_-$ and $F$ in a $T^2$ action on $\bbF$. Note that we have glued the boundary of a tubular neighborhood of the negative section of one $4$-ball to the neighborhood of a fiber in the next. Of course, strictly this cannot be done, but since in $\bbF$ the section $S_-$ and a fiber meet in a point, we need only glue together the boundaries of normal $S^1$-bundles restricted over punctured surfaces (genus $0$ in this case), and these restricted bundles are all trivial. Think of the boundary of our $4$-ball in $\bbF$ as
\[ \bd B^4 = (D^2\x S^1) \cup (T^2\x I) \cup (D^2\x S^1)  \]
where we are gluing together the pieces $D^2\x S^1$, one $4$-ball to the next. This process is illustrated in Figure~\ref{3fold}. In that figure, the wedge-shaped regions in the center denote neighborhoods of the intersection of $S_-$ and $F$ after the section and fiber have been deleted.

\begin{figure}[ht]
\begin{center}\includegraphics[scale=1.25]{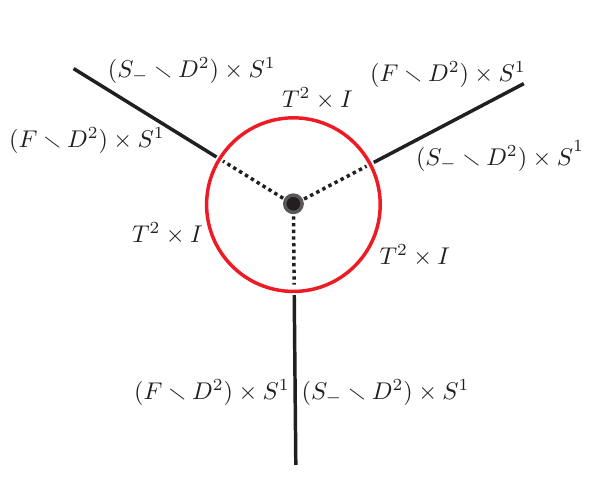}
\end{center}\vspace*{-.1in}
\caption{}
\label{3fold}
\end{figure}

 After gluing all three components this way, we obtain a manifold whose boundary is a $T^2$ bundle over a circle. If the total space of this bundle is $T^3$, then we can fill in the boundary with $T^2\x D^2$ to obtain a closed manifold. Of course this is what happens in our example of $\CP$. More generally, when there are three pieces, if the sum of the two Euler numbers of the normal bundles removed at each interface is $-1$ (as in Figure \ref{CP}(b)), then the boundary will be $T^3$. See \cite{Pin} for more general criteria. The spheres $S_-$ and $F$ are referred to as `interface surfaces'.

Each of the three $4$-ball components comprising the pinwheel structure of $\CP$ has the handlebody decomposition of Figure \ref{Hopf}(a). Removing the negative section $S_-$ from $\bbF$ leaves a $D^2$ bundle over $S^2$ with Euler number $1$ --- this gives the $2$-handle, and removing the fiber $F$ adds the $1$-handle. Figure \ref{Hopf}(b) shows normal circles $\mu_{S_-}$ and $\mu_F$ to the negative section and fiber of $\bbF$.

\begin{figure}[ht]
\begin{center}
    \subfigure[]{\includegraphics[scale=2]{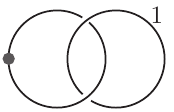}}\hspace{.25in}
    \subfigure[]{\includegraphics[scale=2]{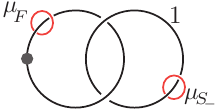}}      
    \end{center}
    \caption{}
  \label{Hopf}
\end{figure} 

The components of the pinwheel are glued together so that $\mu_F$ in one component is identified with $\mu_{S_-}$ in the next. Thus,  $\mu_F$ bounds a disk $D$ (the cocore of the $+1$-framed $2$-handle) in an adjacent pinwheel component. It also follows that the Bing double of $\mu_F$ bounds disjoint disks inside a small neighborhood of $D$. We can ambiently add $2$-handles to the Bing double of $\mu_F$ in the first pinwheel component while subtracting these $2$-handles from the second component. Since subtracting a $2$-handle is equivalent to attaching a $1$-handle, this handle-trading process, when done in each pinwheel component, turns each $4$-ball of Figure \ref{Hopf}(a) into the manifold shown in Figure \ref{InP2}.

\begin{figure}[ht]
\begin{center}\includegraphics[scale=1.2]{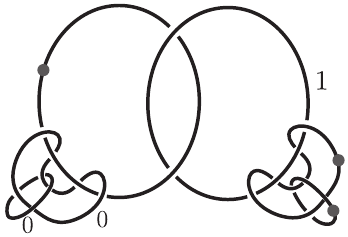}
\end{center}\vspace*{-.1in}
\caption{}
\label{InP2}
\end{figure}

The result of this handle-trading is to create a new pinwheel structure for $\CP$ where the interface surfaces are now tori rather than spheres. This process is shown schematically in Figure \ref{HT}.

\begin{figure}[ht]
\begin{center}\includegraphics[scale=.3]{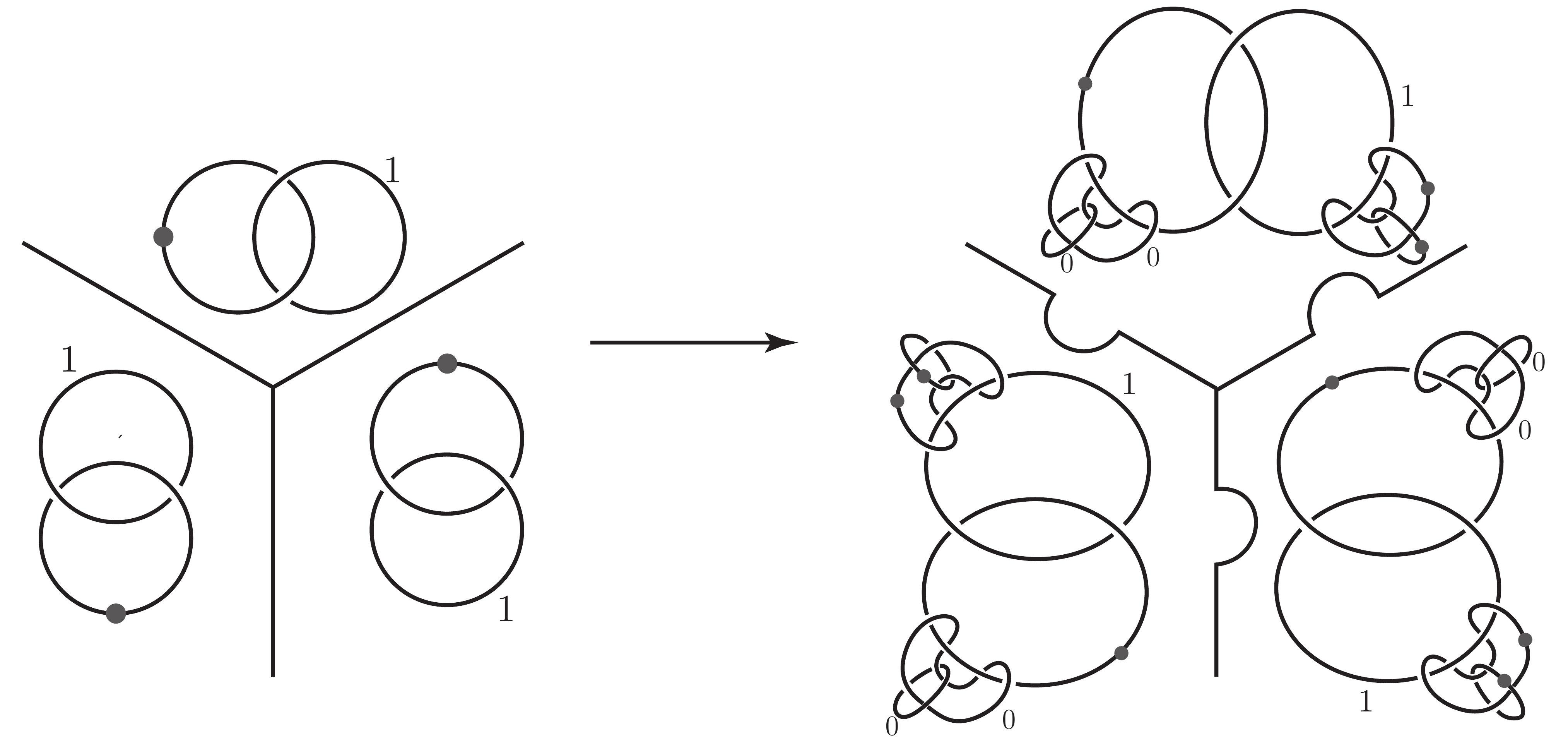}
\end{center}\vspace*{-.1in}
\caption{Handle-trading in $\CP$}
\label{HT}
\end{figure}

Unfortunately, these new pinwheel components do not contain copies of our manifold $A$. (Compare Figures \ref{InP2} and \ref{A}(b).) However, if we blow up each pinwheel component, we get pinwheel components $C_i$, $i=0,1,2$, as illustrated in Figure~\ref{BlownUp}, where all the handles are labeled. 
The new components $C_i$ clearly contain $A$. Blowing up three times accomplishes this for each pinwheel component. Thus $\CPC$ contains three copies of $A$, hence three copies of $B_T$ --- all together six nullhomologous tori. 

We need to see that surgery on these tori produces a symplectic manifold $X$ homeomorphic to $\CPC$ with $K_X\cdot\o_X>0$.  According to Proposition \ref{T_0->A}, standard surgeries on $B_T$ change $A$ into $T_0\x T_0$. So if we perform these surgeries in each pinwheel component, they become
\[ (T_0\x T_0) \# \CPb = (T^2\x T^2)\# \CPb \- (T^2\x \{ pt\}\cup \{ pt\}\x T^2)\]
Three of these get glued together in a pinwheel which is obtained from $\CPC$ via six standard surgeries. Since each component is obtained by removing a a pair of transversely intersecting symplectic surfaces from a symplectic manifold, it follows from a theorem of Symington \cite{S} that it gives a pinwheel structure for  a symplectic manifold $Q$. In \cite{S} this is called a {\it 3-fold sum} of symplectic manifolds. (Surely, $Q$ is $Sym^2(\Sig_3)$, but this is as yet unproved.) Furthermore, the tori $T$ and $S_T$ of Lemma~\ref{SxT} become Lagrangian in $Q$ as in Proposition \ref{T_0->A}.

\begin{figure}
\begin{center}\includegraphics[scale=1.5]{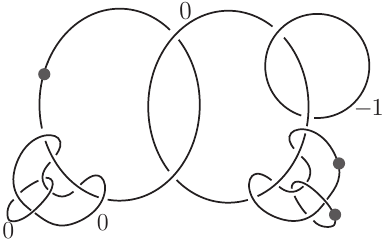}
\end{center}\vspace*{-.1in}
\caption{}
\label{BlownUp}
\end{figure}

Thus we have six Lagrangian tori in $Q$ upon which we can perform Luttinger surgeries to kill the first betti number of $Q$. In \cite{Pin} we showed that the result $X$ of these surgeries is simply connected. (One should not underestimate the importance of this calculation. Many mistakes in the literature have been made at this point.)
Hence, $X$ is the symplectic manifold that we have sought. Let $L_{1,i}$ and $L_{2,i}$ denote the Lagrangian tori in $Q$ obtained from the tori  $T_{(1),i}$ and $T_{(2),i}$ of the $i$th copy of $B_T$ ($i=1,2,3$) after the standard surgeries on $\CPC$. Hence 
\[ \CPC\- \cup_{i=1}^{3}(T_{1,i} \cup T_{2,i})
= Q\- \cup_{i=1}^{3}(L_{1,i} \cup L_{2,i}) = Z,\] say. 
This means that $X$ is the union of $Z$ with six copies of $T^2\x D^2$. That is to say, $X$ is the result of six surgeries on the nullhomologous tori $T_{(1),i}$ and $T_{(2),i}$ in $\CPC$. Combining this with our discussion from \S \ref{RE} gives:

\begin{thm}[\cite{Pin}] There are six nullhomologous tori embedded in $\CPC$ upon which surgery gives rise to an infinite family of mutually nondiffeomorphic $4$-manifolds homeomorphic to $\CPC$.\qed
\end{thm}

\section{Reducing the number of tori}

The goal of this section is to prove Theorem \ref{one}, i.e. that we can find a {\it single} nullhomologous torus in $\CPC$ upon which surgery gives rise to an infinite family of mutually nondiffeomorphic $4$-manifolds homeomorphic to $\CPC$. To do so, we will show that surgery on less than all six of the nullhomologous tori  $T_{i,j}$ in $\CPC$ gives back $\CPC$ again; so its effect is to re-embed the remaining tori. In particular, surgery on any five of these tori will re-embed the remaining torus in $\CPC$, and surgery on this torus will give rise to the distinct smooth structures on $\CPC$. 

The key tool is the following simple lemma. To describe the situation, let $T$ be a self-intersection $0$ torus in a $4$-manifold $X$, and let $b$ be a loop on $T$. Let $S^1_b$ be a loop on the boundary $\bd N_T\cong T^3$ of the tubular neighborhood of $T$ which is homologous to $b$ in $N_T$. The choice of the loop $S^1_b$ trivializes the normal bundle of $T$ restricted over $b$, and this together with a trivialization of the normal bundle of $b$ in $T$ gives us a trivialization of the normal bundle of $S^1_b$ in $\bd N_T$.
Write this trivialized normal bundle as $S^1_b\x D^2\C \bd N_T$. 

 If $S^1_b$ bounds an embedded disk $\DD \C X\- Int (N_T)$ with a neighborhood $\DD\x D^2 \C X\- Int (N_T)$ such that $(\DD\x D^2) \cap \bd N_T=S^1_b\x D^2$ (with the framings agreeing on the boundary), we say that {\em $b$ has a $0$-vanishing cycle}, and we call $\DD$ a {\em $0$-framed disk for $b$}. Figure~\ref{van} gives a handlebody picture of the neighborhood of a self-intersection $0$ torus with a $0$-vanishing cycle.
\begin{figure}[ht]
\begin{center}\includegraphics[scale=1.75]{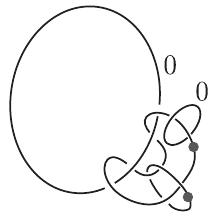}
\end{center}\vspace*{-.1in}
\caption{}
\label{van}
\end{figure}

\begin{lem} \label{0van} Let $T$ be a self-intersection $0$ torus in a $4$-manifold $X$, and let $b$ be a loop on $T$ which has a $0$-vanishing cycle. If $n$ is a nonzero integer, the result of $1/n$ surgery on $T$ with respect to $b$ is again $X$.
\end{lem}
\begin{proof}  Let $W= N_T\cup_\iota (\DD\x D^2)$. The gluing $\iota: S^1_b\x D^2\to\bd\DD\x D^2$ is given by the $0$-framing of $S^1_b$ on $T$ and the normal framing of $T$ given by the fact that $T$ has self-intersection $0$ in $X$. Thus $\iota(t,z)=(t,z)$. It is not difficult to see that $W$ is diffeomorphic to $D^2\x S^2\,\# \,S^1\x D^3$. The result of $1/n$-surgery on $T$ with respect to $b$ and its framing $S^1_b$ is 
\begin{equation}\label{surg} T^2\x D^2 \cup_\vp S^1_a \x S^1_b\x \bd D^2 \x I \cup_\iota  (\DD\x D^2)\end{equation}
where $S^1_a$ is some framing circle for $a$, a complementary circle to $b$ on $T$, where $\bd D^2\x I$ is a collar on $\bd D^2$, and where 
\[ \vp: \bd(T^2\x D^2) = S^1_a \x S^1_b\x \bd D^2\to S^1_a \x S^1_b\x \bd D^2\x\{0\}\]
is $\vp(s,t,z)=(s,tz^n,z)$.  Thus \eqref{surg} becomes 
$ T^2\x D^2 \cup_\vt (\DD\x D^2)$
where $\vt: S^1_b\x D^2\to\bd\DD\x D^2$ is given by $\vt (t,z)=(tz^n,z)$. Hence
\[
\begin{array}{ccc}
T^2\x D^2  & \cup_\iota  &  (\DD\x D^2)  \\
\downarrow{Id} &   & \downarrow{\Theta} \\
T^2\x D^2  & \cup_\vt  &  (\DD\x D^2) 
\end{array}
\]
where $\Theta(w,z)= (wz^n,z)$. We need to be able to extend via the identity on the rest of $X$, and this can be done because $\Theta$ is isotopic to the identity on $\bd W\cap (\DD\x D^2) = \DD\x \bd D^2$.
\end{proof}

\begin{figure}[ht]
\begin{center}\includegraphics[scale=.4]{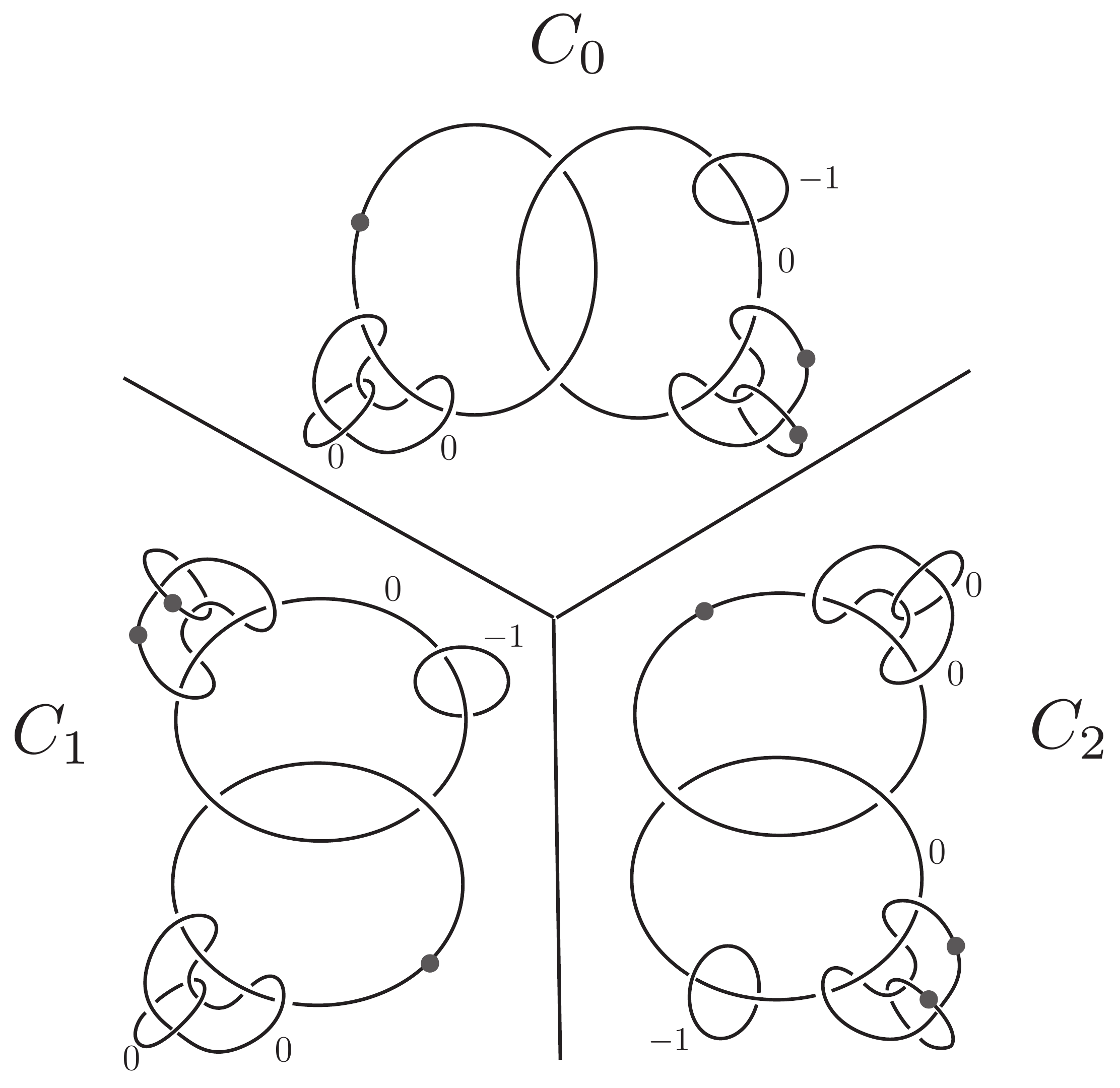}
\end{center}\vspace*{-.1in}
\caption{}
\label{Around}
\end{figure}

Label the pinwheel components of $\CPC$ counterclockwise as $C_i$ ($i=0,1,2$) as in Figure~\ref{Around}; so each $C_i$ has the handlebody description of Figure \ref{ALabelled}. We will use the notation from that figure with the obvious modification that the cocores of the Borromean $2$-handles of $C_i$ are $x_i$ and $y_i$, etc., and the Bing double circles for $b_i$ are $b_{i,1}^*$ and $b_{i,2}^*$.

Suppose that we plan to surger the Bing tori $B_{T,1}$ in $C_1$ and $B_{T,2}$ in $C_2$. The loops $a_2$ and $b_2$,  of $C_2$ are identified with the 
cocores of the Bing doubled $2$-handles, $x_1$ and $y_1$, in $C_1$; so these loops bound $0$-framed disks in $C_1$. Therefore, the Bing doubles $b_{2,1}^*$
and $b_{2,2}^*$ of $b_2$ bound disjoint $0$-framed disks $D^*_{2,i}$ (in a neighborhood of the $0$-framed disk bounded by $b_2$).

This means that the Bing tori $B_{T,2}$ satisfy the hypothesis of Lemma~\ref{0van} with respect to $\pm 1$ surgeries on the $b_{2,i}^*$. Thus there is a diffeomorphism $\Theta_2$ from $\CPC$ to the result of these two surgeries, and $\Theta_2$ has support in the union $C_2\cup D^*_{2,1}\cup D^*_{2,2}$. Notice that the disks 
$D^*_{2,1}$ and $D^*_{2,2}$ intersect the Bing tori in $C_1$ since they intersect 
$D_{x,1}\cup D_{y,1}$. (See Lemma~\ref{DescribeBT}.)

Nonetheless, $\Theta_2(B_{T,1})$ is a pair of Bing tori in $\Theta_2(\CPC)\cong\CPC$. The loops $a_1$ and $b_1$ do not lie in the support of $\Theta_2$, and their normal circles bound disks in $C_0$ analogously to the argument above.  Thus the tori of 
$\Theta_2(B_{T,1})$ satisfy the hypothesis of Lemma~\ref{0van} with respect to $\pm 1$ surgeries on the $b_{1,i}^*$, and there is a diffeomorphism $\Theta_1$ from $\Theta_2(\CPC)$ to the 
result of these two surgeries on $B_{T,1}$. 

We are left with the Bing tori $\Theta_1\Theta_2(B_{T,0})=\Theta_1(B_{T,0})$. Doing $\pm 1$ surgery on one of these tori as before re-embeds the other into a Whitehead double torus in $\CPC$. Thus after five surgeries, we still have $\CPC$, and this proves Theorem~\ref{one}.

Notice that the loops $a_0$ and $b_0$ are not in the support of the diffeomorphism $\Theta_1\Theta_2$, and that their Bing doubles $b_{0,1}^*$
and $b_{0,2}^*$ again bound disjoint $0$-framed disks 
$\Theta_1\Theta_2(D^*_{0,i})$ in  $\Theta_1\Theta_2(\CPC)\cong \CPC$. One can ask why this does not mean that the Bing tori $\Theta_1\Theta_2(B_{T,0})$ satisfy the hypothesis of Lemma~\ref{0van}. The answer is that $\Theta_1\Theta_2$ has moved these disks all the way around Figure~\ref{Around}, and in fact they must intersect $\Theta_1\Theta_2(B_{T,0})=\Theta_1(B_{T,0})$.

We note that although we have proved the existence of our single nullhomologous torus $T$ in $\CPC$, we do not identify it explicitly. It would be interesting to do so (but it is not clear that it would be useful).

Finally, consider the complement $W$ of a tubular neighborhood of our torus $T$ in $\CPC$ as a symplectic manifold with boundary. This manifold admits two very different symplectic structures. The first, with `negative canonical class', extends over $\CPC$, the other, with `positive canonical class', extends over the exotic symplectic manifold $X$.

\end{document}